\documentclass[11pt,leqno,amscd, amsfonts, amssymb,pstricks,verbatim]{amsart}
\usepackage{}
\usepackage{bbm}

\usepackage{mathrsfs}
\usepackage{amsmath}
\usepackage{amssymb}
\usepackage{hyperref}

\usepackage{mathrsfs,color}

\newtheorem{thm}{Theorem}[section]
\newtheorem{lem}[thm]{Lemma}

\newtheorem{prop}[thm]{Proposition}

 \theoremstyle{definition}

\newtheorem{defn}[thm]{Definition}

\newtheorem{rem}[thm]{Remark}
\newtheorem{rems}[thm]{Remarks}

\numberwithin{equation}{thm}

\def\ggg{\mathfrak{g}}
\def\sss{\mathfrak{s}}
\def\lll{\mathfrak{l}}
\def\ppp{\mathfrak{p}}
\def\nnn{\mathfrak{n}}
\def\hhh{\mathfrak{h}}

\def\ad{{\rm ad}}
\def\ch{{\rm ch}}
\def\tr{{\rm tr}}
\def\Hom{{\rm Hom}}
\def\chom{{{\mathcal{H}}\rm om}}
\def\Ext{{\rm Ext}}

\def\Soc{{\rm Soc}}

\def\bba{\mathbb A}
\def\A{{\mathfrak{A}}}
\def\bbf{\mathbb F}
\def\bbz{\mathbb Z}

\def\co{\mathcal O}

\begin{document}

\title[Tilting modules and character formulas]{The Category $\mathcal{O}$  for Lie algebras of vector fields (I): Tilting modules and character formulas}

\author{Fei-Fei Duan, Bin Shu and Yu-Feng Yao}

\address{College of Mathematics and Information Science,
Hebei Normal University, Shijiazhuang, Hebei, 050024, China.}\email{duanfeifei0918@126.com}

\address{Department of Mathematics, East China Normal University,
 Shanghai, 200241, China.} \email{bshu@math.ecnu.edu.cn}

\address{Department of Mathematics, Shanghai Maritime University,
 Shanghai, 201306, China.}\email{yfyao@shmtu.edu.cn}



\thanks{This work is partially supported by the National Natural Science Foundation of China (Grant Nos. 11771279, 11671138, and 11601116), the Natural Science Foundation of Shanghai (Grant No. 16ZR1415000), Shanghai Key Laboratory of PMMP (No. 13dz2260400), and the Scientific Research Foundation of Hebei Education Department(No. QN2017090).}

\begin{abstract}
In this article, we exploit the theory of graded module category with semi-infinite character developed by Soergel in \cite{Soe} to study  representations of the infinite dimensional Lie algebras of vector fields $W(n), S(n)$ and $H(n)$ $(n\geq 2)$, and obtain the description of indecomposable tilting modules. The character formulas  for those tilting modules are determined.
\end{abstract}

\maketitle

\section{Introduction}

Associated with an affine algebraic variety $\mathcal{X}$, the Lie algebras of vector fields on $\mathcal{X}$ are basic algebraic objects. When considering the fundamental case $\mathcal{X}=\bba^n$, we have the Lie algebras of vector fields $W(n)$, $S(n)$, $H(n)$ and $K(n)$. Those Lie algebras are  involved in the classification of transitive Lie pseudogroup raised by E. Cartan (cf. \cite{GS}, \cite{Ko}, \cite{KoN}, \cite{SS}, etc.), and also involved in the classification of finite dimensional simple Lie algebras over an algebraically closed filed of prime characteristic (cf. \cite{Kac} and \cite{KS}).

The present paper is the first one of series papers with which we will focus the concern on the representations of $W(n)$, $S(n)$ and $H(n)$ in an analogy of the BGG category for complex semisimple Lie algebras over an algebraically closed field $\bbf$ of characteristic $0$.

Recall that Rudakov studied the irreducible modules of height grater than one over $\widehat{W(n)}$, the derivation algebras of the formal power series in $n$ variables over complex numbers,  along with the other two series of infinite dimensional Lie algebras of Cartan type, i.e., $\widehat{S(n)}$ and $\widehat{H(n)}$, deciding all those irreducible modules (cf. \cite{Rud}, \cite{Rud2}). Parallel to Rudakov's work, Guang-Yu Shen in \cite{Shen} determined all irreducible graded modules for $W(n)$, $S(n)$ and $H(n)$ over complex numbers with aid of his mixed-product methods. What is more, he completely determined all irreducible graded modules for the finite dimensional counterpart over an algebraically closed field of positive characteristic by the same methods.

In the present paper, the motivation is to extend the above-mentioned arguments of $W(n)$, $S(n)$ and $H(n)$ from the point of view of highest weight category, which enable us to find the connection with the classical theory (for example the BGG category of complex semisimple Lie algebras \cite{Hum}). Recall that the infinite dimensional Lie algebra $X(n)$, $X\in\{W,S,H\}$, is endowed with a canonical graded structure
$$X(n)=\sum_{i=-1}^\infty X(n)_{[i]}$$
arising from the gradation of polynomials from  $\bbf[x_1,\cdots,x_n]$, the coordinate ring of $\bba^n_\bbf$. As homogeneous spaces, $X(n)_{[-1]}$ is $\bbf$-spanned by all partial derivations $\partial_i$, $i=1,\cdots,n$, and $X(n)_{[0]}$ is isomorphic to $\mathfrak{gl}(n)$, $\mathfrak{sl}(n)$ or $\mathfrak{sp}(n)$, containing a canonical maximal torus $\hhh$. We consider the subalgebra $B=X(n)_{[-1]}\oplus X(n)_{[0]}$.
Associated with $B$, we introduce a subcategory $\co$ of $X(n)$-module category, an analogue of the BGG category over complex semi-simple Lie algebras, whose objects satisfy the axioms (see Definition \ref{category C}), in the same spirit as in the BGG category. In the classical theory, there are classes of the canonical objects, including the simple ones, standard ones, co-standard ones, and tilting ones. In our category, the simple objects coincide with the ones studied by Shen and Rudakov.  Especially, we will exploit the related theory developed by Sogergel in \cite{Soe} to our case (assuming $n\geq 2$), and obtain the tilting modules and their characters.

According to Cartan's  classification of transitive Lie pseudogroup (cf. \cite{GS}, \cite{Ko}, \cite{KoN}, \cite{SS}, etc.), there is another type of infinite dimensional contact Lie algebras (type $K$) whose structures are quite different from the other three types. We will study their tilting modules  somewhere else.

\section{Preliminaries}
In this paper, we always assume that the ground field $\mathbb{F}$ is algebraically closed, and of characteristic $0$. All vector spaces (modules) are over $\mathbb{F}$.

\subsection{The Lie algebras of vector fields $W(n)$, $S(n)$ and $H(n)$}\label{CartanType}
Let $n$ be a positive integer, and $P_n=\mathbb{F}[x_1,\cdots, x_n]$ be the polynomial algebra of $n$ indeterminants.  Denote by $W(n)$ the Lie algebra of all  derivations  on $P_n$. Then $W(n)$ is a free $P_n$-module with basis $\{\partial_i\mid 1\leq i\leq n\}$, where $\partial_i$ is the partial derivation with respect to $x_i$, i.e., $\partial_i(x_j)=\delta_{ij}$ for $1\leq i,j\leq n$. The natural $\mathbb{Z}$-grading on $P_n$ induces the corresponding $\mathbb{Z}$-grading on $W(n)$, i.e., $W(n)=\bigoplus\limits_{i=-1}^{\infty}W(n)_{[i]}$, where $W(n)_{[i]}=\text{span}_{\mathbb{F}}\{f_i\partial_i\mid 1\leq i\leq n, f_i\in P_n, \text{deg}(f_i)=i+1\}$.

The Lie algebra $S(n)$ of special type is a subalgebra of $W(n)$ consisting of vector fields $\sum_i f_i\partial_i$ with zero divergence, i.e., $S(n)=\{\sum_i f_i\partial_i\in W(n)\mid \sum_i\partial_i(f_i)=0\}$. By the definition, it is easily seen that $S(n)$ is spanned by those elements $D_{ij}(x^{\alpha})$ with $\alpha=(\alpha(1),\cdots, \alpha(n))\in\mathbb{N}^n$,
$x^{\alpha}=x_1^{\alpha(1)}\cdots x_n^{\alpha(n)}$, and $1\leq i<j\leq n$, where $D_{ij}: P_n\longrightarrow P_n$ is a linear mapping defined by $D_{ij}(x^{\alpha})=
\alpha_jx^{\alpha-\epsilon_j}
\partial_i-\alpha_ix^{\alpha-\epsilon_i}\partial_j,\,\forall\,\alpha\in\mathbb{N}^n$,  with $\epsilon_k:=(\delta_{1k},\cdots,\delta_{nk})$ for $k=1,\cdots,n$. Since the divergence operator is a homogeneous operator of degree 0, the algebra $S(n)$ inherits the $\mathbb{Z}$-gradation of $W(n)$. {\sl{Hereafter, we abuse the notation $x^\alpha$ for $\alpha\in \bbz^n$, by making the convention that $x^\alpha=0$ unless $\alpha\in \mathbb{N}^n$.}}

When $n=2r$ is even, the elements in $W(n)$ that annihilate the 2-form $\sum_{i=1}^r dx_i\wedge dx_{i+r}$ are called Hamiltonian.  The Lie algebra $H(n)$ of Hamiltonian type is a subalgebra of $W(n)$ consisting for all Hamiltonian elements in $W(n)$. By the definition,  $H(n)$ has a canonical basis $\{D_H(x^{\alpha})\mid \alpha\in\mathbb{N}^n\setminus\{(0,\cdots, 0)\}\}$, where $D_H: P_n\longrightarrow P_n$ is a linear mapping defined by $D_H(x^{\alpha})=\sum\limits_{i=1}^n\sigma(i)\partial_i(x^{\alpha})\partial_{i^{\prime}}$ with
\begin{equation*}
\sigma(i)=\begin{cases}
1, &\text{if}\,\, 1\leq i\leq r,\cr
-1, &\text{if}\,\, r\leq i\leq n,
\end{cases}
\end{equation*}
and
\begin{equation*}
i^{\prime}=\begin{cases}
i+r, &\text{if}\,\, 1\leq i\leq r,\cr
i-r, &\text{if}\,\, r\leq i\leq n.
\end{cases}
\end{equation*}
Since the 2-form  $\sum_{i=1}^r dx_i\wedge dx_{i+r}$ can be regarded as an operator of degree 2, the algebra $H(n)$ inherits the $\mathbb{Z}$-gradation of $W(n)$.

In the following, let $\ggg=X(n)$, $X\in\{W,S,H\}$. Then $\ggg$ has a $\mathbb{Z}$-gradation $\ggg=\bigoplus\limits_{i=-1}^{\infty}\ggg_{[i]}$, where $\ggg_{[i]}=\ggg\cap W(n)_{[i]}$ for $i\geq -1$. Let $\ggg_i=\bigoplus_{j\geq i}\ggg_{[j]}$. We then have the following $\mathbb{Z}$-filtration of $\ggg$: $$\ggg=\ggg_{-1}\supset \ggg_0\supset\ggg_1\cdots.$$
It should be noted that
\begin{equation}\label{grading 0}
\ggg_{[0]}\cong\begin{cases}
\ggg\lll(n),&\text{if}\,\, \ggg=W(n),\cr
\sss\lll(n),&\text{if}\,\, \ggg=S(n),\cr
\sss\ppp(n),&\text{if}\,\, \ggg=H(n).
\end{cases}
\end{equation}
We have a triangular decomposition $\ggg_{[0]}=\nnn^-\oplus\hhh\oplus\nnn^+$, where
\begin{equation*}
\nnn^{-}=\begin{cases}
\text{span}_{\mathbb{F}}\{x_i\partial_j\mid 1\leq j<i\leq n\}, &\text{if}\,\, \ggg=W(n), S(n),\cr
\text{span}_{\mathbb{F}}\{x_i\partial_j-x_{j+r}\partial_{i+r}, x_{s+r}\partial_{t}+x_{t+r}\partial_{s}\mid 1\leq j<i\leq r, 1\leq s\leq t\leq r\}, &\text{if}\,\, \ggg=H(2r),
\end{cases}
\end{equation*}

\begin{equation*}
\hhh=\begin{cases}
\text{span}_{\mathbb{F}}\{x_i\partial_i\mid 1\leq i\leq n\}, &\text{if}\,\, \ggg=W(n),\cr
\text{span}_{\mathbb{F}}\{x_i\partial_i-x_j\partial_j\mid 1\leq i<j\leq n\}, &\text{if}\,\, \ggg=S(n),\cr
\text{span}_{\mathbb{F}}\{x_i\partial_i-x_{i+r}\partial_{i+r}\mid 1\leq i\leq r\}, &\text{if}\,\, \ggg=H(2r),
\end{cases}
\end{equation*}
and
\begin{equation*}
\nnn^{+}=\begin{cases}
\text{span}_{\mathbb{F}}\{x_i\partial_j\mid 1\leq i<j\leq n\}, &\text{if}\,\, \ggg=W(n), S(n),\cr
\text{span}_{\mathbb{F}}\{x_i\partial_j-x_{j+r}\partial_{i+r}, x_s\partial_{t+r}+x_t\partial_{s+r}\mid 1\leq i<j\leq r, 1\leq s\leq t\leq r\}, &\text{if}\,\, \ggg=H(2r).
\end{cases}
\end{equation*}
The negative root system associated with $\nnn^-$  is denoted by $\Phi^-$.
Let $B=\ggg_{\leq 0}:=\ggg_{[-1]}\oplus\ggg_{[0]}$, and $U(B), U(\ggg)$ be the universal enveloping algebra of $B$ and $\ggg$, respectively.  The $\mathbb{Z}$-gradation on $\ggg$ (resp. $B$) induces a natural $\mathbb{Z}$-gradation on $U(\ggg)$ (resp. $U(B)$), i.e., $U(\ggg)=\bigoplus_{i\in\mathbb{Z}}U(\ggg)_{[i]}$ (resp. $U(B)=\bigoplus_{i\in\mathbb{Z}}U(B)_{[i]}$).

\subsection{Semi-infinite characters} \label{semi-inf}
In general, for a $\bbz$-graded Lie algebra $\ggg=\sum_{i\in\bbz} \ggg_{[i]}$  with $\dim \ggg_{[i]}<\infty$ for all $i\in\bbz$. A character $\gamma: \ggg_{[0]} \rightarrow \bbf$ is called a semi-infinite character for $\ggg$ if the following items satisfy
\begin{itemize}
\item[(SI-1)] As a Lie algebra, $\ggg$ is generated by $\ggg_{[1]}, \ggg_{[0]}$  and $\ggg_{[-1]}$;
\item[(SI-2)] $\gamma([X,Y]) = \tr\big((\ad X \,\ad Y)|_{\ggg_{[0]}})$, $\forall\, X\in\ggg_{[1]}$ and $Y\in \ggg_{[-1]}$.
\end{itemize}

Now we have the following basic observation.
\begin{lem}\label{semi-inf}
Assume $\ggg=X(n), X\in\{W,S,H\}$ with $n\geq 2$. Let $\mathcal{E}_W:\ggg_{[0]}\longrightarrow\mathbb{F}$ be a linear map with $\mathcal{E}_W(x_i\partial_j)=\delta_{ij}$ for $1\leq i,j\leq n$. Let $\mathcal{E}_S=\mathcal{E}_H=0$. Then $\mathcal{E}_X$ is a semi-infinite character for $X\in\{W,S,H\}$.
\end{lem}

\begin{proof} Obviously, $\mathcal{E}_X$ is a homomorphism of  Lie algebras from $\ggg_{[0]}$ to the trivial Lie algebra $\bbf$. By a straightforward calculation, it is readily shown that (SI-2)  satisfies for all $X(n)$, $X\in\{W,S,H\}$. We proceed to check (SI-1). Denote by  $\overline\ggg$ the Lie subalgebra generated by $\ggg_{[i]}$ for $i$ running through $\{-1,0,1\}$. We will show $\overline\ggg$ coincides with $\ggg$ case by case.
More precisely, we will show that the homogeneous space $\ggg_{[i]}$ is contained in $\overline\ggg$ for all $i\geq 2$. To this end, we only need to check that
\begin{align}\label{check}
\ggg_{[i]} \mbox{ is included in  }[\ggg_{[i-1]},\ggg_{[1]}],
\end{align}
then it is consequently concluded by induction that $\ggg_{[i]}$ is contained in $\overline\ggg$ for all $i\geq 2$.

(1) $\ggg=W(n)$. Now we take any basis element $X:=x^\alpha\partial_k$ in $\ggg_{[i]}$ ($i\geq 2$) with $\alpha=\sum_{j=1}^n\alpha(j)\epsilon_j$. Then $i\geq 2$ implies that $|\alpha|:=\sum_{j}\alpha(j)>2$. We show (\ref{check}), dividing into different cases.

Case 1: $\alpha=\alpha(k)\epsilon_k$.

We take $l\in \{1,\cdots,n\}$ with $l\ne k$ because $n\geq 2$. Then
\begin{equation*}
x^\alpha\partial_k=\begin{cases}
[x_k^2\partial_l, x_kx_l\partial_k+x_l^2\partial_l]\in [\ggg_{[i-1]},\ggg_{[1]}], &\text{if}\,\, \alpha(k)=3,\cr
\frac{1}{3-\alpha(k)}[x_k^{\alpha(k)-1}\partial_k, x_k^2\partial_k]\in [\ggg_{[i-1]},\ggg_{[1]}], &\text{if}\,\,\alpha(k)>3.
\end{cases}
\end{equation*}

Case 2: $\alpha\ne \alpha(k)\epsilon_k$.

In this case, there exists some $l\in \{1,\cdots,n\}$ with $l\ne k$ such that $\alpha(l)\ne 0$. Then
\begin{equation*}
x^\alpha\partial_k=\begin{cases}
{1\over2} [x^{\alpha-\epsilon_k}\partial_k,x_k^2\partial_k]\in [\ggg_{[i-1]},\ggg_{[1]}], &\text{if}\,\, \alpha(k)=1,\cr
{1\over{1-\alpha(k)}} [x^{\alpha-\epsilon_l}\partial_k,x_kx_l\partial_k]\in [\ggg_{[i-1]},\ggg_{[1]}], &\text{if}\,\,\alpha(k)\ne 1.
\end{cases}
\end{equation*}

(2) $\ggg=S(n)$. We take any generating element $X:=x^{\beta}\partial_k$ or $D_{kl}(x^{\alpha})$ in $\ggg_{[i]}$ ($i\geq 2$) with $\beta=\sum_{j=1}^n\beta(j)\epsilon_j$ and $\alpha=\sum_{j=1}^n\alpha(j)\epsilon_j$  satisfying  $\beta(k)=0$, while $\alpha(k)\neq 0$ and $\alpha(l)\neq 0$.  We show (\ref{check}), dividing into three cases.

Case 1: $X=x^{\beta}\partial_k$ with $\beta(k)=0$.

In this case, we can take $j\in \{1,\cdots,n\}$ with $j\ne k$ and $\beta(j)>0$. Then
$$x^{\beta}\partial_k=\frac{1}{\beta(j)+1}[D_{jk}(x^{\beta-\epsilon_j+\epsilon_k}), x_j^2\partial_k]\in [\ggg_{[i-1]},\ggg_{[1]}].$$

Case 2: $X=D_{kl}(x^{\alpha})$ with $\alpha(k)\neq 2\alpha(l)+1, \alpha(k)\neq 0$, and $\alpha(l)\neq 0$. (Note that the situation when either $\alpha(l)=0$ or $\alpha(k)=0$ turns into Case 1.)

In this case, we have
$$D_{kl}(x^{\alpha})=\frac{1}{2\alpha(l)+1-\alpha(k)}[D_{kl}(x^{\alpha-\epsilon_k}), D_{kl}(x^{2\epsilon_k+\epsilon_l})]\in [\ggg_{[i-1]},\ggg_{[1]}].$$

Case 3: $X=D_{kl}(x^{\alpha})$ with $\alpha(k)=2\alpha(l)+1$, and $\alpha(l)\neq 0$. (Note that the situation when $\alpha(l)=0$ turns into Case 1.)

In this case, we have
$$D_{kl}(x^{\alpha})=-\frac{1}{3\alpha(l)+3}[D_{kl}(x^{\alpha-\epsilon_l}), D_{kl}(x^{2\epsilon_l+\epsilon_k})]\in [\ggg_{[i-1]},\ggg_{[1]}].$$

(3) $\ggg=H(n)$ with $n=2r$. We take any basis element $X:=D_H(x^{\alpha})$ in $\ggg_{[i]}$ ($i\geq 2$) with $\alpha=\sum_{j=1}^n\alpha(j)\epsilon_j$. We show (\ref{check}), dividing into two cases.

Case 1: There exists some $k$ with $1\leq k\leq n$ such that $\alpha(k)\geq 2$.

In this case, it follows from a straightforward computation that
$$D_H(x^{\alpha})=\frac{1}{3\sigma(k)(\alpha(k^{\prime})+1)}[D_H(x^{3\epsilon_k}), D_H(x^{\alpha-2\epsilon_k+\epsilon_{k^{\prime}}})]\in [\ggg_{[i-1]},\ggg_{[1]}].$$

Case 2: $\alpha(k)=\alpha(l)=1$ for some $k, l$ with $1\leq k<l\leq n$, and $\alpha(t)=0$ or $1$ for all $t\neq k, l$.

In this case, it follows from a straightforward computation that

\begin{equation*}
D_H(x^{\alpha})=\begin{cases}
\frac{1}{2\sigma(l)(\alpha(l^{\prime})+1)}[D_H(x^{\epsilon_k+2\epsilon_l}), D_H(x^{\alpha+\epsilon_{l^{\prime}}-\epsilon_k-\epsilon_l})]-&\cr
\;\;\;\frac{2}{3\sigma(k)}[D_H(x^{3\epsilon_l}), D_H(x^{\alpha+2\epsilon_{l^{\prime}}-\epsilon_l-\epsilon_k-\epsilon_{k^{\prime}}})]\in [\ggg_{[i-1]},\ggg_{[1]}], &\text{if}\,\, \alpha(k^{\prime})=1,\cr
\frac{1}{2\sigma(l)(\alpha(l^{\prime})+1)}[D_H(x^{\epsilon_k+2\epsilon_l}), D_H(x^{\alpha+\epsilon_{l^{\prime}}-\epsilon_k-\epsilon_l})]\in [\ggg_{[i-1]},\ggg_{[1]}], &\text{if}\,\,\alpha(k^{\prime})=0.
\end{cases}
\end{equation*}

Summing up, we have proved (\ref{check}).
\end{proof}

\subsection{The category $\mathcal{O}$}  The following notion is an analogy of the BGG category for complex finite dimensional semi-simple Lie algebras.

\begin{defn}\label{category C}
Denote by $\mathcal{O}$ the category, whose objects are additive groups $M$ with the following three properties satisfied.
\begin{itemize}
\item[(1)] $M$ is an admissible $\mathbb{Z}$-graded $\ggg$-module, i.e., $M=\bigoplus\limits_{i\in\mathbb{Z}} M_{[i]}$ with $\dim M_{[i]}<+\infty$, and
$\ggg_{[i]}M_{[j]}\subseteq M_{[i+j]},\forall\,i,j$.
\item[(2)] $M$ is locally finite for $B$. Here $B=\ggg_{\leq 0}:=\ggg_{[-1]}\oplus\ggg_{[0]}$  is defined  as in \S\,\ref{CartanType}.
\item[(3)] $M$ is $\hhh$-semisimple, i.e., $M$ is a weight module: $M=\bigoplus_{\lambda\in\hhh^*}M_{\lambda}$.
\end{itemize}
The morphisms in $\mathcal{O}$ are the $\ggg$-module morphisms that respect the $\mathbb{Z}$-gradation, i.e.,
$$\Hom_{\mathcal{O}}(M, N)=\{f\in\Hom_{U(\ggg)}(M, N)\mid f(M_{[i]})\subseteq N_{[i]},\,\forall\,i\in\mathbb{Z}\},\,\forall\,M,N\in\mathcal{O}.$$
\end{defn}

\begin{rems}\label{rem for C}
\begin{itemize}
\item[(1)] It is readily shown that the category $\mathcal{O}$ is an abelian category.
\item[(1)] Denote by $\co_{B\mbox{-}{\rm fin}}^{0}$ (resp. $\co_{\ggg_{[0]}\mbox{-}{\rm fin}}^{0}$)  the category of locally finite $B$-modules (resp. $\ggg_{[0]}$-modules). Then any irreducible module $M\in \co_{B\mbox{-}{\rm fin}}^{0}$ is finite dimensional and is a simple module in $\co_{\ggg_{[0]}\mbox{-}{\rm fin}}^{0}$ with trivial $\ggg_{[-1]}$-action, and vice versa.
\item[(2)] It follows from Definition \ref{category C} (ii) that $M|_{\ggg_{[0]}}\in\co_{\ggg_{[0]}\mbox{-}{\rm fin}}^{0}$ for any module $M$ in $\mathcal{O}$. Since $M$ is $\hhh$-semisimple, it is easy to see that $M|_{\ggg_{[0]}}$ is semisimple.
\end{itemize}
\end{rems}

Similar to \cite[Lemma 5.8]{Soe}, we have the following parallel result.

\begin{lem}
There are enough injectives in $\mathcal{O}$.
\end{lem}

\section{Standard modules}
\subsection{}\label{basic notations}
Keep notations as before, in particular, $\ggg=\bigoplus\limits_{i=-1}^{\infty}\ggg_{[i]}$ is one of the Lie algebras of vector fields $W(n), S(n)$ and $H(n)$, and $\hhh$ is the standard Cartan subalgebra of $\ggg_{[0]}$ (recall $\ggg_{[0]}\cong \mathfrak{gl}(n)$ for $W(n)$, $\mathfrak{sl}(n)$ for $S(n)$ and $\mathfrak{sp}(n)$ for $H(n)$ under the isomorphism correspondence $W_{[0]}\rightarrow \mathfrak{gl}(n)$ with $x_i\partial_j\mapsto E_{ij}$). Denote by $\epsilon_i$ the linear function on $\sum_{j=1}^n\bbf x_j\partial_j$ via defining
  $\epsilon_i(x_j\partial_j)=\delta_{ij}$ for $1\leq i,j\leq n$. In the natural sense, we  identify the unit function $\epsilon_i$ with $(\delta_{1i},\cdots, \delta_{ni})$ for $1\leq i\leq n$.  With those unit linear functions, we can express the weight functions that we need for the arguments on $\ggg_{[0]}$-modules in the sequent.
   Let $\Lambda^{-}$ be the set of anti-dominant integral weights relative to the standard Borel subalgebra $\hhh+\nnn^+$ of $\ggg_{[0]}$, which means that $\lambda\in \Lambda^{-}$ if and only if $-\lambda$ is a dominant integral weight in the sense of \cite{Hum1}. Then finite dimensional irreducible $\ggg_{[0]}$-modules are parameterized by $\Lambda^-\times\mathbb{Z}$. For any $\lambda\in \Lambda^-$, let $L_0^{-}(\lambda)_d$ be the simple $\ggg_{[0]}$-module concentrated in a single degree $d$ with the lowest weight $\lambda$.  Set $\Delta(\lambda)_d=U(\ggg)\otimes_{U(B)}L_0^{-}(\lambda)_d$,  where $L_0^{-}(\lambda)_d$ is regarded as a $B$-module with trivial $\ggg_{[-1]}$-action. Then $\{\Delta(\lambda)_d\mid \lambda\in \Lambda^-,d\in\mathbb{Z}\}$ constitute a class of so-called standard modules for $\ggg$ in the usual sense. We have the following result.
\begin{lem}\label{lem1}
Let $\lambda\in\Lambda^-, d\in\mathbb{Z}$. The following statements hold.
\begin{itemize}
\item[(1)] The standard module $\Delta(\lambda)_d$ is an object in $\mathcal{O}$.
\item[(2)] The standard module $\Delta(\lambda)_d$ has a unique irreducible quotient, denoted by $L(\lambda)_d$.
\item[(3)] The iso-classes of irreducible modules in $\mathcal{O}$ are parameterized by $\Lambda^-\times \mathbb{Z}$. More precisely, each simple module $S$ in $\mathcal{O}$ is of the form $L(\mu)$ for some $\mu\in \Omega$ with the depth $d$ of $S$ defined by $S=\sum\limits_{i\geq d}S_{[i]}$, where $S_{[d]}\neq 0$.
\end{itemize}
\end{lem}

\begin{proof}

(1) As a vector space, $\Delta(\lambda)_d=U(\ggg_1)\otimes_{\mathbb{F}} L_0^{-}(\lambda)_d$. Let $v=\sum\limits_{i\in I} v_{i+d}\in\Delta(\lambda)_d$, where $I\subset\mathbb{Z}_{\geq 0}$ is a finite index set, $v_{i+d}=\sum\limits_{s} u_s\otimes w_{i,s}$ with $w_{i,s}\in L_0^{-}(\lambda)_d$ and $\{u_s\}$ is a  basis of $U(\ggg_1)_{[i]}$. To show that $\Delta(\lambda)_d\in\mathcal{O}$, we need to prove that $U(B)v_{i+d}$ are finite-dimensional for any nonzero homogeneous vector $v_{i+d}$. For that, on one hand, since
$$\ggg_{[0]}v_{i+d}=\ggg_{[0]}\sum\limits_{s} u_s\otimes w_{i,s}\subset\sum\limits_{s} U(\ggg_1)_{[i]}\otimes w_{i,s}+\sum\limits_{s} u_s\otimes \ggg_{[0]}w_{i,s}
\subset U(\ggg_1)_{[i]}\otimes L_0^{-}(\lambda),$$
we get $U(\ggg_{[0]})v_{i+d}\in U(\ggg_1)_{[i]}\otimes L_0^{-}(\lambda)$. This implies that $U(\ggg_{[0]})v_{i+d}$ is finite-dimensional. On the other hand, by a similar argument, we have
$$U(B)v_{i+d}\subset U(\ggg_{[-1]})U(\ggg_1)_{[i]}\otimes L_0^{-}(\lambda)\subset \sum\limits_{0\leq s\leq i}U(\ggg_1)_{[s]}\otimes L_0^{-}(\lambda).$$
This implies that $U(B)v_{i+d}$ is finite-dimensional.

(2) Any proper submodule of $\Delta(\lambda)_d$ is contained in $\ggg_{1}U(\ggg_1)\otimes L_0^{-}(\lambda)_d$, so is the sum of all proper submodules of $\Delta(\lambda)_d$. Hence, the sum of all proper submodules is the unique maximal submodule of $\Delta(\lambda)_d$, i.e.,  $\Delta(\lambda)_d$ has a unique simple quotient.

(3) Let $S$ be any irreducible module in $\mathcal{O}$. Since $S$ is $U(B)$ locally finite, we can take a finite dimensional irreducible $U(B)$-submodule $S_1$. It follows from Remark \ref{rem for C} that $\ggg_{[-1]}$ acts trivially on $S_1$, and $S_1$ is a finite dimensional irreducible $U(\ggg_{[0]})$-module. Hence, $S_1$ is isomorphic to $L_0^{-}(\mu)_l$ for some $\mu\in\Lambda^{-}, l\in\mathbb{Z}$. Consequently, $S$ is a quotient of $\Delta(\mu)_l$. Then it follows from the statement (2) that $S\cong L(\mu)_l$. Moreover, since $S=U(\ggg)v=U(\ggg_1)U(B)v$ and $\dim U(B)v<\infty$ for any nonzero $v\in S$, we know that there exists a unique integer $d$ such that $S=\sum\limits_{i\geq d} S_{[i]}$ with $S_{[d]}\neq 0$.
\end{proof}

\begin{rem}
 We usually write $\Delta(\lambda)_0$ (resp. $L_0^{-}(\lambda)_0$, $L(\lambda)_0$) as  $\Delta(\lambda)$ (resp. $L_0^{-}(\lambda)$, $L(\lambda)$) for brevity.
\end{rem}

\subsection{Depths}  An integer $d$ appearing in Lemma \ref{lem1}(3) is  called the depth of $S$. In general, for $M\in\co$ with  $M=\sum\limits_{i\geq d} M_{[i]}$,  we say that $M$ admits  depth $d$ if  $M_{[d]}\neq 0$, but $M_{[j]}=0$ for $j<d$.
Set $\mathcal{O}_{\geq d}:=\{M\in\mathcal{O}\mid M=\sum\limits_{i\geq d}M_{[i]}\}$.

The translation functor $T[-d]:\mathcal{O}_{\geq d}\longrightarrow\mathcal{O}_{\geq 0}$ relates $\mathcal{O}_{\geq d}$ and $\mathcal{O}_{\geq 0}$, so that we only need to focus on $\mathcal{O}_{\geq 0}$ when we make arguments on module structure. In the following, we always assume that $\Delta(\lambda), L(\lambda)$ are objects in $\mathcal{O}_{\geq 0}$, which implies that  $L_0^{-}(\lambda)$ falls in the grading-zero component. (However, the study involving depths is still very import  when we consider the topics related to the iso-classes of irreducible modules, for example the blocks of the category $\co$, which will be seen in our sequent paper.)

\section{Costandard modules and their prolonging realization}
Keep the same notations as in the previous sections. For a $\mathbb{Z}$-graded algebra $\A$ and $\mathbb{Z}$-graded modules $M$ and $N$ over $\A$, define the set of admissible $\A$-homomorphisms as follows.
$$\chom_{\A}(M,N):=\{f\in\Hom_{\A}(M, N)\mid \text{there\,exists\,some}\,i\in\mathbb{Z}\,\text{such\,that}\, f(M_{[j]})=0,\,\forall\,j\neq i\}.$$
\subsection{Costandard modules} Let $\lambda\in\Lambda^-$. Define the costandard $\ggg$-module corresponding to $\lambda$ as
\begin{eqnarray*}
\nabla(\lambda)&:=&\chom_{U(\ggg_0)}(U(
\ggg), L_0^{-}(\lambda))\cr
&=&\{f\in\Hom_{U(\ggg_0)}(U(
\ggg), L_0^{-}(\lambda))\mid \text{there\,exists\,some}\,i\in\mathbb{Z}\,\text{such\,that}\, f(U(\ggg)_{[j]})=0,\,\forall\,j\neq i\},
\end{eqnarray*}
where $L_0^{-}(\lambda)$ is regarded as a $\ggg_{0}$-module with trivial $\ggg_1$-action.
Then for any $i\in\mathbb{Z}$, set
\begin{eqnarray*}
\nabla(\lambda)_{[i]}=\{f\in\Hom_{U(\ggg_0)}(U(
\ggg), L_0^{-}(\lambda))\mid f(U(\ggg)_{[j]})=0,\,\forall\,j\neq -i\}.
\end{eqnarray*}
it is readily known that $\nabla(\lambda)=\bigoplus_{i\in\mathbb{Z}}\nabla(\lambda)_{[i]}$ with $\dim\nabla(\lambda)_{[i]}<\infty$ for any $i\in\mathbb{Z}$, and $\nabla(\lambda)_{[j]}=0$ for any $j<0$. Hence $\nabla(\lambda)\in\mathcal{O}_{\geq 0}$. We have the following result. 
\begin{lem}
Let $\lambda,\mu\in\Lambda^-$, then the following statements hold.
\begin{itemize}
\item[(1)] $L(\lambda)$ admits a projective cover $\Delta(\lambda)$ in $\mathcal{O}_{\geq 0}$.
\item[(2)] $L(\lambda)$ admits an injective hull $\nabla(\lambda)$ in $\mathcal{O}_{\geq 0}$.
\item[(3)] $\Hom_{\mathcal{O}_{\geq 0}}(\Delta(\lambda), \nabla(\mu))=0$ if $\lambda\neq \mu$.
\item[(4)] $\Ext_{\mathcal{O}_{\geq 0}}^1(\Delta(\lambda), \nabla(\mu))=0$ for any $\lambda,\mu$.
\end{itemize}
\end{lem}

\begin{proof}

(1) Take any $M\in\mathcal{O}_{\geq 0}$, we then have
\begin{eqnarray*}
\Hom_{\mathcal{O}_{\geq 0}}(\Delta(\lambda), M)&=&\Hom_{\mathcal{O}_{\geq 0}}(U(\ggg)\otimes_{U(B)} L_0^{-}(\lambda), M)\\
&=&\Hom_{U(\ggg_{[0]})}(L_0^{-}(\lambda), M_{[0]}).
\end{eqnarray*}
This implies that $\Delta(\lambda)$ is projective in $\mathcal{O}_{\geq 0}$ by Remarks\ref{rem for C}, since $\dim M_{[0]}<\infty$.
Moreover, it follows from the proof of Lemma \ref{lem1}(ii) that $\Delta(\lambda)$ has a unique maximal submodule. Hence, $\Delta(\lambda)$ is indecomposable. Moreover, $\dim\Hom_{\mathcal{O}_{\geq 0}}(\Delta(\lambda), \Delta(\lambda))<\infty$ and $\mathcal{O}_{\geq 0}$ is an abelian category. Consequently, $\Delta(\lambda)$ is the projective cover of $L(\lambda)$ in $\mathcal{O}_{\geq 0}$ by \cite[Lemma 3.3]{Soe}.

(2) We note that
\begin{eqnarray*}
\nabla(\lambda)&=&\chom_{U(\ggg_0)}(U(\ggg),L_0^{-}(\lambda))\\
&\cong&\chom_{U(\ggg_{[0]})}(U(B),L_0^{-}(\lambda))\\
&\cong&\chom_{U(\ggg_{[0]})}(U(B),\mathbb{F})\otimes_{\mathbb{F}}L_0^{-}(\lambda).
\end{eqnarray*}
Hence, as a $B$-module, the socle of $\nabla(\lambda)$ is isomorphic to $L_0^{-}(\lambda)$. Consequently, $\Soc_{U(\ggg)}(\nabla(\lambda))\cong L(\lambda)$. Moreover, for any $M\in\mathcal{O}_{\geq 0}$, we have
\begin{eqnarray*}
\Hom_{\mathcal{O}_{\geq 0}}(M, \nabla(\lambda))&=&\Hom_{\mathcal{O}_{\geq 0}}(M, \chom_{U(\ggg_0)}(U(\ggg),L_0^{-}(\lambda)))\\
&=&\chom_{U(\ggg_{0})}(M,L_0^{-}(\lambda))\\
&=&\chom_{U(\ggg_{[0]})}(M,L_0^{-}(\lambda)).
\end{eqnarray*}
It follows that $\nabla(\lambda)$ is injective in $\mathcal{O}_{\geq 0}$. Hence, $\nabla(\lambda)$ is the injective hull of $L(\lambda)$.

(3) \begin{eqnarray*}
\Hom_{\mathcal{O}_{\geq 0}}(\Delta(\lambda), \nabla(\mu))&=&\Hom_{\mathcal{O}_{\geq 0}}(U(\ggg)\otimes_{U(B)}L_0^{-}(\lambda),\nabla(\mu))\\
&=&\Hom_{U(\ggg_{[0]})}(L_0^{-}(\lambda),L_0^{-}(\mu))\\
&=&\delta_{\lambda\mu}.
\end{eqnarray*}

(4) It follows form the statement (1), since $\Delta(\lambda)$ is projective in $\mathcal{O}_{\geq 0}$.
\end{proof}

\subsection{Prolonging realization} In this subsection, we introduce  a kind realization of costandard modules $\nabla(\lambda)$ for $\lambda\in\Lambda^-$ via prolonging $L_0^{-}(\lambda)$ as below. Set $\mathcal{V}(\lambda)=P_n\otimes L_0^{-}(\lambda)$ for $\lambda\in\Lambda^-$. It follows from \cite[Theorem 2.1]{Skr}  that we can endow with a $W(n)$-module structure $\rho_{_{W(n)}}$ on $\mathcal{V}(\lambda)$ via
\begin{equation}\label{W(n)-module structure}
\rho_{_{W(n)}}(\sum\limits_{i=1}^n f_i\partial_i)(g\otimes v)=\sum\limits_{i=1}^n f_i(\partial_i(g))\otimes v+\sum\limits_{i=1}^n\sum\limits_{j=1}^n (\partial_j(f_i))g\otimes \xi(x_j\partial_i)v
\end{equation}
for any $f_i,g\in P_n, v\in L_0^{-}(\lambda)$,
where $\xi$ is the representation of the $W(n)_{[0]}$-module $L_0^{-}(\lambda)$. Furthermore, it is a routine to check that we have a $\ggg$-module structure
on $\mathcal{V}(\lambda)$ for $\ggg\in \{S(n), H(n)\}$,
via:
\begin{eqnarray}\label{S(n)-module}
\rho_{\ggg}(D_{kl}(x^{\alpha})) (g\otimes v)&=&(D_{kl}(x^{\alpha}))(g)\otimes v+\alpha(k)\alpha(l)x^{\alpha-\epsilon_k-\epsilon_l}g\otimes\xi(x_k\partial_k-x_l\partial_l)v\\
&&+\sum\limits_{\stackrel{j=1}{j\neq k}}^n
\alpha(l)(\alpha(j)-\delta_{jl})x^{\alpha-\epsilon_j-\epsilon_l}g\otimes\xi(x_j\partial_k)v\nonumber\\
&&-\sum\limits_{\stackrel{j=1}{j\neq l}}^n
\alpha(k)(\alpha(j)-\delta_{jk})x^{\alpha-\epsilon_j-\epsilon_k}g\otimes\xi(x_j\partial_l)v,\nonumber
\end{eqnarray}
and
\begin{eqnarray}\label{H(n)-module}
\rho_{\ggg}(D_H(x^{\alpha})) (g\otimes v)&=&(D_H(x^{\alpha}))(g)\otimes v+\sum\limits_{j=1}^{2r}\sigma(j)\alpha(j)(\alpha(j)-1)x^{\alpha-2\epsilon_j}g\otimes \xi (x_{j}\partial_{j^{\prime}})v\\
&&+\sum\limits_{1\leq j<k\leq r}\alpha(j)\alpha(k)x^{\alpha-\epsilon_j-\epsilon_k}g\otimes\xi(x_k\partial_{j^{\prime}}+x_j\partial_{k^{\prime}})v\nonumber\\
&&-\sum\limits_{k=1}^r\sum\limits_{j=r+1}^{2r}\alpha(j)\alpha(k)x^{\alpha-\epsilon_j-\epsilon_k}g\otimes\xi(x_k\partial_{j^{\prime}}-
x_j\partial_{k^{\prime}})v\nonumber\\
&&-\sum\limits_{r+1\leq j<k\leq 2r}\alpha(j)\alpha(k)x^{\alpha-\epsilon_j-\epsilon_k}g\otimes\xi(x_k\partial_{j^{\prime}}+x_j\partial_{k^{\prime}})v\nonumber
\end{eqnarray}
for $\ggg=S(n), H(n)$ respectively. Here $\alpha\in\mathbb{N}^n, 1\leq k<l\leq n, g\in P_n, v\in  L_0^{-}(\lambda)$, $\xi$ is the representation of the $\ggg_{[0]}$-module $L_0^{-}(\lambda)$.

\begin{rem} A conceptual account for the above $\ggg$-module structure can be provided by \cite[Theorem 1.2]{Shen1}.
\end{rem}

The following result asserts that the costandard
$\ggg$-module $\nabla(\lambda)$ is isomorphic to $\mathcal{V}(\lambda)$ for $\ggg=X(n)$, $X\in\{W,S,H\}$.

\begin{prop}\label{costandard iso}
Keep the notations as above, then $\nabla(\lambda)\cong \mathcal{V}(\lambda)$ as $U(\ggg)$-modules.
\end{prop}

\begin{proof}
We just need to prove the assertion for $\ggg=W(n)$. Similar arguments yield the statements for $\ggg=S(n)$ and $H(n)$. Below, we assume $\ggg=W(n)$.
As vector spaces,
\begin{eqnarray*}
\nabla(\lambda)&=&\Hom_{U(\ggg_0)}(U(\ggg), L_0^{-}(\lambda))\\
&\cong &\Hom_{U(\ggg_0)}(U(\ggg),\mathbb{F})\otimes_{\mathbb{F}} L_0^{-}(\lambda)\\
&\cong &\Hom_{\mathbb{F}}(U(\ggg_{[-1]}),\mathbb{F})\otimes_{\mathbb{F}} L_0^{-}(\lambda).
\end{eqnarray*}
Let
\begin{eqnarray*}
\psi:\,\mathcal{V}(\lambda)& \longrightarrow &\nabla(\lambda)\\
f\otimes v&\longmapsto&\psi(f\otimes v),\,\,f\in P_n, v\in L_0^{-}(\lambda),
\end{eqnarray*}
where $\psi(f\otimes v)\in \nabla(\lambda)$ is defined as
$$\psi(f\otimes v)(\partial_1^{a_1}\cdots\partial_n^{a_n})=\big(\partial_1^{a_1}\cdots\partial_n^{a_n}(f)|_{(x_1,\cdots, x_n)=(0,\cdots, 0)}\big)v.$$
Note that $\nabla(\lambda)$ is spanned by $\{f_{\underline{q}, v}\mid \underline{q}=(q_1,\cdots, q_n)\in\mathbb{N}^n, v\in L_0^{-}(\lambda)\}$, where
$$f_{\underline{q}, v}(\partial_1^{b_1}\cdots\partial_n^{b_n})=\delta_{\underline{b},\underline{q}}v,\,\underline{b}=(b_1,\cdots, b_n), \underline{q}=(q_1,\cdots, q_n)\in\mathbb{N}^n, v\in L_0^{-}(\lambda).$$
Then $\psi$ is surjective, since $f_{\underline{q}, v}$ is the image of $cx_1^{q_1}\cdots x_n^{q_n}\otimes v$ under the map $\psi$ for some nonzero constant $c$. Moreover, $\psi$ is injective. Indeed, if $\psi(\sum\limits_{i=1}^s f_i\otimes v_i)=0$, where $f_i\in P_n, v_i\in L_0^{-}(\lambda)$, and $\{v_i\}$ are linearly independent, then $\partial_1^{b_1}\cdots \partial_n^{b_n}(f)|_{(x_1,\cdots, x_n)=(0,\cdots, 0)}=0$ for any $(b_1,\cdots, b_n)\in\mathbb{N}^n$. This implies that $f_i=0$, $\forall\,i$, i.e., $\sum f_i\otimes v_i=0$.

In the following, we show that $\psi$ is a $\ggg$-module homomorphism. For that, we first remind the readers of  the multiplied combinatorial number ${\underline{a}\choose\underline{b}}=\prod_{i=1}^n{a_i\choose b_i}$   for $\underline{a}=(a_1,\cdots, a_n), \underline{b}=(b_1,\cdots, b_n)\in \mathbb{N}^n$, where ${s\choose t}:={s!\over t!(s-t)!}$ if $s\geq t$, $0$ if $s<t$, and ${s\choose 0}:=0$.
Taking any $\alpha,\beta, \underline{a}=(a_1,\cdots, a_n)\in \mathbb{N}^n$, $1\leq s \leq n$, $v\in L_0^{-}(\lambda)$, we have
\begin{eqnarray*}
\psi(x^{\alpha}\partial_s\cdot (x^{\beta}\otimes v))(\partial_1^{a_1}\cdots \partial_n^{a_n})&=&
\psi\big(x^{\alpha}(\partial_s(x^{\beta}))\otimes v+\sum\limits_{j=1}^nx^{\alpha-\epsilon_j}x^{\beta}\otimes E_{js}v\big)(\partial_1^{a_1}\cdots\partial_n^{a_n})\\
&=&\big(\partial_1^{a_1}\cdots\partial_n^{a_n}(x^{\alpha}x^{\beta-\epsilon_s})|_{\underline{x}=\underline{0}}\big) v+\sum\limits_{j=1}^n\big(\partial_1^{a_1}\cdots\partial_n^{a_n}(x^{\alpha-\epsilon_j}x^{\beta})|_{\underline{x}=\underline{0}}\big) E_{js}v\\
&=&\bigg(\sum\limits_{\underline{t}}{\underline{a}\choose\underline{t}}x^{\alpha-\underline{t}}x^{\beta-\epsilon_s-(\underline{a}-\underline{t})}
|_{\underline{x}=\underline{0}}\bigg)v\\
&&+\sum\limits_{j=1}^n\sum\limits_{\underline{u}}\bigg({\underline{a}\choose \underline{u}}x^{\alpha-\epsilon_j-\underline{u}}x^{\beta-(\underline{a}-\underline{u})}|_{\underline{x}=\underline{0}}\bigg) E_{js}v\\
&=&{\underline{a}\choose \alpha}\delta_{\alpha+\beta, \underline{a}+\epsilon_s}v+\sum\limits_{j=1}^n{\underline{a}\choose \alpha-\epsilon_j}\delta_{\alpha+\beta,\underline{a}+\epsilon_j}E_{js}v.
\end{eqnarray*}
On the other hand,
\begin{eqnarray*}
&&\big(x^{\alpha}\partial_s\cdot \psi(x^{\beta}\otimes v)\big)(\partial_1^{a_1}\cdots\partial_n^{a_n})\\
&= &\psi(x^{\beta}\otimes v)(\partial_1^{a_1}\cdots\partial_n^{a_n}(x^{\alpha}\partial_s))\\
&= &\psi(x^{\beta}\otimes v)\bigg(\sum\limits_{\underline{k}}{\underline{a}\choose \underline{k}}x^{\alpha+\underline{k}-\underline{a}}\partial_s\partial_1^{k_1}\cdots\partial_n^{k_n}\bigg)\\
&=&{\underline{a}\choose \underline{a}-\alpha}\delta_{\beta,\underline{a}-\alpha+\epsilon_s}v+\sum\limits_{j=1}^n{\underline{a}\choose \underline{a}-\alpha+\epsilon_j}\delta_{\beta,\underline{a}-\alpha+\epsilon_j}E_{js}v\\
&=&{\underline{a}\choose \alpha}\delta_{\alpha+\beta,\underline{a}+\epsilon_s}v+\sum\limits_{j=1}^n{\underline{a}\choose \alpha-\epsilon_j}\delta_{\alpha+\beta,\underline{a}+\epsilon_j}E_{js}v.
\end{eqnarray*}
Hence
$$\psi(x^{\alpha}\partial_s\cdot (x^{\beta}\otimes v))=x^{\alpha}\partial_s\cdot \psi(x^{\beta}\otimes v).$$
This implies that $\psi$ is a $\ggg$-module isomorphism, as desired. We complete the proof.
\end{proof}

\subsection{}  Recall the notations in \S\ref{basic notations} for unit linear functions. We have the following definition of exceptional weights for further use.
\begin{defn}
Let $\ggg=X(n), X\in\{W,S,H\}$, be a Lie algebra of vector fields. Set $\omega_0=0$ and
\begin{equation*}
\omega_k=\begin{cases}
\epsilon_{n+1-k}+\cdots+\epsilon_n,&\text{if}\,\,X\in\{W,S\},\cr
-\epsilon_1-\cdots-\epsilon_k,&\text{if}\,\,X=H,
\end{cases}
\end{equation*}
for $1\leq k\leq n^{\prime}$, where
\begin{equation*}
n^{\prime}=\begin{cases}
n,&\text{if}\,\, X=W,\cr
n-1,&\text{if}\,\,X=S,\cr
\frac{n}{2},&\text{if}\,\,X=H.
\end{cases}
\end{equation*}
These $\omega_k\,(0\leq k\leq n^{\prime})$ are called exceptional weights. The corresponding simple $\ggg$-modules $L(\omega_k)$ ($0\leq k\leq n^{\prime}$) are called exceptional $\ggg$-modules.
\end{defn}

The following result is due to A. Rudakov and G. Shen.

\begin{prop}(\cite[Theorem 13.7, and Corollaries 13.8-13.9]{Rud}, \cite[Theorem 4.8]{Rud2} and \cite[Theorem 2.4]{Shen})\label{known result1}
Let $\ggg=W(n)$ or $S(n)$. Then the following statements hold.
\begin{itemize}
\item[(1)] If $\lambda\in\Lambda^-$ is not exceptional, then $\mathcal{V}(\lambda)$ is a simple $\ggg$-module.
\item[(2)] The following sequence
$$0\longrightarrow\mathcal{V}(\omega_0)\xrightarrow{\,\,\,d_0\,\,\,}\mathcal{V}(\omega_1)\xrightarrow{\,\,\,d_1\,\,\,}\cdots\cdots \mathcal{V}(\omega_k)\xrightarrow{\,\,\,d_k\,\,\,}\mathcal{V}(\omega_{k+1})\xrightarrow{d_{k+1}}\cdots\cdots \mathcal{V}(\omega_{n-1})\xrightarrow{d_{n-1}} \mathcal{V}(\omega_{n})\longrightarrow 0$$
\end{itemize}
is exact, where
\begin{eqnarray*}
d_k:\,\mathcal{V}(\omega_k)& \longrightarrow &\mathcal{V}(\omega_{k+1})\\
x^{\alpha}\otimes (x_{j_1}\wedge\cdots\wedge x_{j_k})&\longmapsto&\sum\limits_{i=1}^n\partial_i(x^{\alpha})\otimes (x_{j_1}\wedge\cdots\wedge x_{j_k}\wedge x_i),\,\forall\,\alpha\in {\mathbb{N}}^n, 1\leq j_1<\cdots < j_k\leq n.
\end{eqnarray*}
\item[(3)] For $0\leq k\leq n-1$, $\mathcal{V}(\omega_k)$ contains two composition factors $L(\omega_k)$ and $L(\omega_{k+1})$ with free multiplicity. And $\mathcal{V}(\omega_n)\cong L(\omega_n)$.
\end{prop}

\begin{prop}(\cite[Theorem 5.10]{Rud2} and \cite[Theorem 2.5]{Shen})\label{known result2}
Let $\ggg=H(n)$, $n=2r$. Then the following statements hold.
\begin{itemize}
\item[(1)] If $\lambda\in\Lambda^-$ is not exceptional, then $\mathcal{V}(\lambda)$ is a simple $\ggg$-module.
\item[(2)] The composition factors of $\mathcal{V}(\omega_k)$ are $L(\omega_{k-1}), L(\omega_{k})$ and $L(\omega_{k+1})$ with
$[\mathcal{V}(\omega_k):L(\omega_{k-1})]=[\mathcal{V}(\omega_k):L(\omega_{k+1})]=1$ and $[\mathcal{V}(\omega_k):L(\omega_{k})]=2$, $0\leq k\leq r$, where we make convention that $L(\omega_{-1})=0$.
\end{itemize}
\end{prop}

\begin{rem} There is a modular version of propositions \ref{known result1}, \ref{known result2} (cf. \cite[Theorems 2.1, 2.2, 2.3]{Shen}).
\end{rem}

\section{Tilting modules and character formulas}
In the concluding section, we determine the character formulas for tilting modules in $\mathcal{O}$. Keep the notations as previously.
\subsection{Tilting modules}
\begin{defn}
An object $M\in\mathcal{O}$ is said to admit a $\Delta$-flag if there exists an increasing filtration
$$0=M_0\subset M_1\subset M_2\subset\cdots\cdots$$
such that $M=\bigcup\limits_{i=0}^{\infty}M_i$ and $M_{i+1}/M_{i}\cong \Delta(\lambda_i)$ for all $i\geq 1$, where $\lambda_i\in\Lambda^-,\,\forall\,i$.
\end{defn}

The following result follows from \cite[Theorem 5.2]{Soe}.

\begin{prop}(\cite[Theorem 5.2]{Soe})\label{tilting}
For each $\lambda\in\Lambda^-$, up to isomorphism, there exists a unique indecomposable object $T(\lambda)\in\mathcal{O}$ such that
\begin{itemize}
\item[(1)] $\Ext_{\mathcal{O}}^1(\Delta(\mu), T(\lambda))=0,\,\forall\,\mu\in\Lambda^-$.
\item[(2)] $T(\lambda)$ admits a $\Delta$-flag, starting with $\Delta(\lambda)$ at the bottom.
\end{itemize}
\end{prop}

The indecomposable module $T(\lambda)$ in Proposition \ref{tilting} is called the tilting module corresponding to $\lambda\in\Lambda^-$.
Now we are in position to further study the tilting modules by Soergel's theory.

Owing to Lemma \ref{semi-inf},  we have the following consequence from \cite[Theorem 5.12]{Soe} and Proposition \ref{costandard iso}.
\begin{prop}\label{soegrel formula}
Let $\ggg=X(n), X\in\{W,S,H\}$ with $n\geq 2$. Let $\lambda, \mu\in\Lambda^-$. Then we have
$$[T(\lambda):\Delta(\mu)]=[\mathcal{V}(-w_0\mu-\mathcal{E}_X):L(-w_0\lambda-\mathcal{E}_X)],$$
where $w_0$ is the longest element in the Weyl group of $\ggg_{[0]}$.
\end{prop}

\begin{rem} Recall  the notations in \S\ref{basic notations} for unit linear functions.
It should be noted that
\begin{equation*}
w_0\epsilon_i=\begin{cases}
\epsilon_{n+1-i},&\text{if}\,\, \ggg=W(n)\,\text{or}\,S(n),\cr
-\epsilon_i,&\text{if}\,\,\ggg=H(n),
\end{cases}
\end{equation*}
for any $1\leq i\leq n$.
\end{rem}

In the following, we always assume $n\geq 2$. We will precisely determine the multiplicity of the standard module $\Delta(\mu)$ occurring in the $\Delta$-filtration of the tilting module $T(\lambda)$ for any $\lambda, \mu\in\Lambda^-$.

\begin{prop}\label{main thm1}
Let $\ggg=W(n)$ with $n\geq 2$, $\lambda,\mu\in\Lambda^-$. The following statements hold.
\begin{itemize}
\item[(1)] If $\mu=-2\sum\limits_{i=1}^k\epsilon_i-\sum\limits_{j=k+1}^n\epsilon_j$ for some $k$ with $0\leq k\leq n-1$, then $[T(\lambda):\Delta(\mu)]\neq 0$ if and only if $\lambda=\mu$ or $\lambda=\mu-\epsilon_{k+1}$. Moreover, $[T(\mu):\Delta(\mu)]=[T(\mu-\epsilon_{k+1}):\Delta(\mu)]=1$.
\item[(2)] If $\mu\neq -2\sum\limits_{i=1}^k\epsilon_i-\sum\limits_{j=k+1}^n\epsilon_j$ for any $0\leq k\leq n-1$, then $[T(\lambda):\Delta(\mu)]\neq 0$ if and only if $\lambda=\mu$. Moreover,  $[T(\mu):\Delta(\mu)]=1$.
\end{itemize}
\end{prop}

\begin{proof}
It follows from Proposition \ref{soegrel formula} that $[T(\lambda):\Delta(\mu)]=[\mathcal{V}(-w_0\mu-\mathcal{E}):L(-w_0\lambda-\mathcal{E})]$, where
$\mathcal{E}=\sum\limits_{i=1}^n\epsilon_i$.

Case 1: $\mathcal{V}(-w_0\mu-\mathcal{E})$ is not simple.

In this case, it follows from Proposition \ref{known result1} that $-w_0\mu-\mathcal{E}=\omega_k=\epsilon_{n-k+1}+\cdots+\epsilon_n$ for some $0\leq k\leq n-1$, i.e., $\mu=-2\sum\limits_{i=1}^k\epsilon_i-\sum\limits_{j=k+1}^n\epsilon_j$. Then by Proposition \ref{known result1}, $[\mathcal{V}(-w_0\mu-\mathcal{E}):L(-w_0\lambda-\mathcal{E})]\neq 0$ if and only if $-w_0\lambda-\mathcal{E}=\omega_k$ or $\omega_{k+1}$. This implies that $\lambda=\mu$ or $\lambda=\mu-\epsilon_{k+1}$, and $[T(\mu):\Delta(\mu)]=[T(\mu-\epsilon_{k+1}):\Delta(\mu)]=1$.

Case 2: $\mathcal{V}(-w_0\mu-\mathcal{E})$ is simple.

In this case, it follows from Proposition \ref{known result1} that $\mu\neq -2\sum\limits_{i=1}^k\epsilon_i-\sum\limits_{j=k+1}^n\epsilon_j$ for any $0\leq k\leq n-1$. The remaining assertions are obvious.
\end{proof}

Similar arguments as in the proof of Theorem \ref{main thm1} yield the following two results for $S(n)$ and $H(n)$.

\begin{prop}\label{main thm2}
Let $\ggg=S(n)$, $\lambda,\mu\in\Lambda^-$, then the following statements hold.
\begin{itemize}
\item[(1)] If $\mu=-\sum\limits_{i=1}^k\epsilon_i$ for some $k$ with $0\leq k\leq n-1$, then $[T(\lambda):\Delta(\mu)]\neq 0$ if and only if $\lambda=\mu$ or
$\lambda=\mu-\epsilon_{k+1}$. Moreover, $[T(\mu):\Delta(\mu)]=[T(\mu-\epsilon_{k+1}):\Delta(\mu)]=1$.
\item[(2)] If $\mu\neq -\sum\limits_{i=1}^k\epsilon_i$ for any $0\leq k\leq n-1$, then $[T(\lambda):\Delta(\mu)]\neq 0$ if and only if $\lambda=\mu$. Moreover,  $[T(\mu):\Delta(\mu)]=1$.
\end{itemize}
\end{prop}

\begin{prop}\label{main thm3}
Let $\ggg=H(n)$, $n=2r$. Let $\lambda,\mu\in\Lambda^-$, then the following statements hold.
\begin{itemize}
\item[(1)] If $\mu=\omega_k$ for $0\leq k\leq r$, then $[T(\lambda):\Delta(\mu)]\neq 0$ if and only if $\lambda=\mu, \mu+\epsilon_k$ or
$\mu-\epsilon_{k+1}$. Moreover, $[T(\mu+\epsilon_k):\Delta(\mu)]=[T(\mu-\epsilon_{k+1}):\Delta(\mu)]=1$, and $[T(\mu):\Delta(\mu)]=2$.
\item[(2)] If $\mu\neq \omega_k$ for any $0\leq k\leq r$, then $[T(\lambda):\Delta(\mu)]\neq 0$ if and only if $\lambda=\mu$. Moreover,  $[T(\mu):\Delta(\mu)]=1$.
\end{itemize}
\end{prop}

\subsection{Character formulas} Now we introduce the formal characters of modules from $\co$. For this, let $\Delta_1$ be the root system of $\ggg_1$ relative to $\hhh$, i.e., $\ggg_1=\sum\limits_{\alpha\in\Delta_1}\ggg_{\alpha}$ with
$\ggg_{\alpha}=\{x\in\ggg\mid [h,x]=\alpha(h)x,\,\,\forall\,h\in\hhh\}$. We next introduce a subset $D(\lambda)\subset \hhh^*$ associated with $\lambda\in \Lambda^-$: $D(\lambda)=\{\mu\in \hhh^*\mid \mu\succeq \lambda\}$, where $\mu\succeq \lambda$ means that $\mu-\lambda$ lies in the $\bbz_{\geq 0}$-span of $\Delta_1\cup \Phi^-$. Then we define an $\bbf$-algebra $\mathcal{A}$, whose elements are series of the form $\sum_{\lambda\in \hhh^*}c_\lambda e^\lambda$ with $c_\lambda\in \bbf$ and $c_\lambda=0$ for $\lambda$ outside the union of a finite number of sets of the form $D(\mu)$. Then $\mathcal{A}$ naturally becomes a commutative associative algebra if we define $e^\lambda e^\mu=e^{\lambda+\mu}$, and identify $e^0$ with the identity element. All formal exponentials $\{e^\lambda\}$ are linearly independent, and then in one-to-one correspondence with $\hhh^*$. For a semisimple $\hhh$-module $W=\sum_{\lambda\in\hhh^*}W_\lambda$, if  the weight spaces are all finite-dimensional, then we can define $\ch(W)=\sum_{\lambda\in\hhh^*}\dim W_\lambda e^\lambda$. In particular, if $V$ is an object in $\mathcal{O}$,
then $\ch(V)\in \mathcal{A}$. We have the following obvious fact.

\begin{lem}\label{cha fact}  The following statements hold.
\begin{itemize}
\item[(1)] Let $V_1, V_2$ and $V_3$ be three $\ggg$-modules in the category $\co$. If there is an exact sequence of $\ggg$-modules $0\rightarrow V_1\rightarrow V_2\rightarrow V_3\rightarrow 0 $, then $\ch(V_2)=\ch(V_1)+\ch(V_3)$.
\item[(2)] Suppose $W=\sum_{\lambda\in\hhh^*}W_\lambda$ is a semi-simple $\hhh$-module with finite-dimensional weight spaces, and $U=\sum_{\lambda\in\hhh^*}U_\lambda$ is a finite-dimensional $\hhh$-module. If $\ch(W)=\sum_{\lambda\in\hhh^*}c_{\lambda}e^\lambda$ falls in $\mathcal{A}$, then $\ch(W\otimes_\bbf U)$ must fall in $\mathcal{A}$ and $\ch(W\otimes_\bbf U)=\ch(W)\ch(U)$.
\end{itemize}
\end{lem}

Let us investigate the formal character of a standard module $\Delta(\lambda)$ for $\lambda\in \Lambda^-$.
Recall $\Delta(\lambda)=U(\ggg_1)\otimes_{\bbf}L_0^{-}(\lambda)$. As a $U(\ggg_1)$-module, $\Delta(\lambda)$ is a free module of rank $\dim L_0^-(\lambda)$ generated by $L_0^-(\lambda)$. By Lemma \ref{cha fact}(2), we have $\ch(\Delta(\lambda))=\ch(U(\ggg_1))\ch L_0^-(\lambda)$ for $\lambda\in \Lambda^-$.  Set
$$\Pi=\prod\limits_{\alpha\in\Delta_1}
(1-e^{\alpha})^{-1},$$
then we further have $\ch(\Delta(\lambda))=\Pi\ch L_0^-(\lambda)$.
As a direct consequence of Propositions \ref{main thm1}, \ref{main thm2} and \ref{main thm3} along with Lemma \ref{cha fact}, we have the following consequences on character formulas for tilting modules.

\begin{thm}
Let $\ggg=W(n)$ and $\lambda\in\Lambda^-$ ($n\geq 2$). The following statements hold.
\begin{itemize}
\item[(1)] If $\lambda=-2\sum\limits_{i=1}^k\epsilon_i-\sum\limits_{j=k+1}^n\epsilon_j$ for some $k$ with $1\leq k\leq n$, then
$$\ch(T(\lambda))=\Pi(\ch(L_0^{-}(\lambda))+\ch(L_0^{-}(\lambda+\epsilon_k))).$$
\item[(2)] If $\lambda\neq -2\sum\limits_{i=1}^k\epsilon_i-\sum\limits_{j=k+1}^n\epsilon_j$ for any $k$ with $1\leq k\leq n$, then
$$\ch(T(\lambda))=\Pi(\ch(L_0^{-}(\lambda))).$$
\end{itemize}
\end{thm}

\begin{thm}
Let $\ggg=S(n)$ and $\lambda\in\Lambda^-$, then the following statements hold.
\begin{itemize}
\item[(1)] If $\lambda=-\sum\limits_{i=1}^k\epsilon_i$ for some $k$ with $1\leq k\leq n-1$, then
$$\ch(T(\lambda))=\Pi(\ch(L_0^{-}(\lambda))+\ch(L_0^{-}(\lambda+\epsilon_k))).$$
\item[(2)] If $\lambda\neq -\sum\limits_{i=1}^k\epsilon_i$ for any $0\leq k\leq n-1$, then
$$\ch(T(\lambda))=\Pi(\ch(L_0^{-}(\lambda))).$$
\end{itemize}
\end{thm}

\begin{thm}
Let $\ggg=H(n)$, $n=2r$, and $\lambda\in\Lambda^-$, then the following statements hold.
\begin{itemize}
\item[(1)] If $\lambda=\omega_k$ for $0\leq k\leq r$, then
$$\ch(T(\lambda))=\Pi(\ch(L_0^{-}(\omega_{k-1}))+2\ch(L_0^{-}(\omega_k))+\ch(L_0^{-}(\omega_{k+1}))).$$
\item[(2)] If $\lambda\neq \omega_k$ for any $0\leq k\leq r$, then
$$\ch(T(\lambda))=\Pi(\ch(L_0^{-}(\lambda))).$$
\end{itemize}
\end{thm}

\subsection*{Acknowledgements}
The authors would like to thank the referee for his/her helpful comments and suggestion. Y.F. Yao thanks Hao Chang for helpful discussions.

\end{document}